\documentclass{amsart}
\usepackage{graphicx}
\usepackage{amsmath}
\usepackage{amssymb}
\usepackage{color}
\usepackage{amsfonts}
\usepackage{mathrsfs}
\newtheorem{theorem}{Theorem}[section]
\newtheorem{corollary}[theorem]{Corollary}

\newtheorem{lemma}[theorem]{Lemma}
\newtheorem{proposition}[theorem]{Proposition}
\theoremstyle{definition}
\newtheorem{e-definition}[theorem]{Definition}
\theoremstyle{remark}
\newtheorem{remark}[theorem]{Remark}

\numberwithin{equation}{section}


\begin{document}
\title{Wiener algebra for the quaternions} %
\author[D. Alpay]{Daniel Alpay}
\address{(DA) Department of Mathematics\\
Ben-Gurion University of the Negev\\
Beer-Sheva 84105 Israel} \email{dany@math.bgu.ac.il}

\author[F. Colombo]{Fabrizio Colombo}
\address{(FC) Politecnico di
Milano\\Dipartimento di Matematica\\Via E. Bonardi, 9\\20133
Milano, Italy}
\email{fabrizio.colombo@polimi.it}

\author[D. P. Kimsey]{David P. Kimsey}
\address{(DPK) Department of Mathematics\\
Ben-Gurion University of the Negev\\
Beer-Sheva 84105 Israel}
\email{dpkimsey@gmail.com}

\author[I. Sabadini]{Irene Sabadini}
\address{(IS) Politecnico di
Milano\\Dipartimento di Matematica\\Via E. Bonardi, 9\\20133
Milano, Italy}
\email{irene.sabadini@polimi.it}

\thanks{D. Alpay thanks the Earl Katz family for endowing the chair
which supported his research. D.P. Kimsey gratefully acknowledges the support of a Kreitman postdoctoral fellowship. F. Colombo and I. Sabadini
acknowledge the Center for Advanced Studies of the Mathematical
Department of the Ben-Gurion University of the Negev for the
support and the kind hospitality during the period in which part
of this paper has been written.}

\date{November 19, 2014}%

\begin{abstract}
We define and study the counterpart of the Wiener algebra in the quaternionic setting, both for the discrete and continuous case. We prove
a Wiener-L\'evy type theorem and a factorization theorem. We give applications to Toeplitz and Wiener-Hopf operators.
\end{abstract}

\maketitle

\section{Introduction and preliminaries}
\setcounter{equation}{0}
The Wiener algebra $\mathcal W^{n\times n}$ consists of the functions of the form
$f(t)=\sum_{u\in\mathbb Z}f_ue^{iut}$
where $f_u\in\mathbb C^{n\times n}$ and
$\sum_{u\in\mathbb Z}\|f_u\|<\infty$, where $\|\cdot\|$ denotes the operator norm.
The space $\mathcal W^{n\times n}$ with pointwise multiplication
and the norm defined above is a Banach algebra. We denote by
$\mathcal W_+^{n\times n}$ (resp. $\mathcal W_-^{n\times n}$) the
subalgebra of functions $f \in \mathcal W^{n\times n}$ for which $f_u=0$ for $u<0$ (resp.
$f_u=0$ for $u>0$). When $n=1$ we will denote the Wiener algebra by $\mathcal W$.\\

The Wiener algebra is important both in harmonic and complex analysis.
The celebrated Wiener-L\'evy theorem
characterizes the invertible elements of $\mathcal W^{n\times n}$: an element is invertible in $\mathcal W^{n\times n}$ if and only if it is
pointwise invertible (in $\mathbb C^{n\times n}$). For $n=1$, this result was first proved by Wiener, see \cite{wiener}. It is also a consequence of the theory of commutative Banach algebras, see \cite{rudin}, and for the case of coefficients in a normed ring see \cite{bp}.
In this paper we discuss how to generalize the Wiener algebra to the quaternionic setting, and we prove an analogue of the Wiener-L\'evy theorem and the Wiener-Hopf factorization. The fact that the quaternions are a skew field requires tools different from those used in the complex case.
 Here we consider the case of quaternionic-valued functions, while the case of quaternionic matrix-valued functions
will require other methods. \\

In the sequel, by $\mathbb H$ we denote the real algebra of quaternions. A quaternion $p$ is an element of the form $p=x_0+x_1e_1+x_2e_2+x_3e_3$ where $e_1^2=e_2^2=e_1e_2e_3=-1$. We will denote by $\mathbb S$ the sphere of unitary purely imaginary quaternions. Any element $i\in\mathbb S$ is such that $i^2=-1$ and any two orthogonal elements $i,j\in\mathbb S$ form a new basis $i,j, ij$ of $\mathbb H$. Given a quaternion $p_0$, it determines a sphere $[p_0]$ consisting of all the points of the form $q^{-1}p_0q$ for $q\not=0$.\\
A quaternionic valued function $f(p)$ belongs to the (discrete) quaternionic Wiener algebra $\mathcal W_{\mathbb H}$ if it is of the form $\sum_{u\in\mathbb Z} p^u f_u$ where $\sum_{u\in\mathbb Z} |f_u|<\infty$. The sum of two elements in $\mathcal W_{\mathbb H}$ is defined in the natural way while their product, denoted by $\star$, is obtained by taking the convolution of the coefficients (see \cite{fliess} and Section 2). In order to relate the  $\star$-product of two elements in $\mathcal W_{\mathbb H}$ with the pointwise product  it is necessary to introduce a suitable map, that we denote by $\omega$ which allows to associate to the values of a quaternionic function to a $2\times 2$ matrix with complex entries. This map is suitably defined by using the more classical map $\chi$ which, once that one fixes an imaginary unit $i$ and another imaginary unit, say $j$, orthogonal to $i$,  associates to $p=z+wj$ the matrix
$$
\begin{pmatrix} z & w\\
-\overline{w}& \overline{z}\end{pmatrix}.
$$
The new map $\omega$ is crucial to characterize the invertible elements in the Wiener algebra, see Theorem \ref{wiener}. It is interesting to note a specific feature of the class of functions we are considering:  one of the equivalent conditions in this result is required only on one specific complex plane (so on a subset of $\mathbb H$) while another condition is required to hold globally. A similar result can be proved also for the subalgebra $\mathcal W_{\mathbb H, +}$ (resp. $\mathcal W_{\mathbb H, -}$) of elements such that $f_u=0$ for $u<0$ (resp. $u>0$).\\
The definition of quaternionic Wiener algebra in the continuous case is somewhat more complicated, see Section 4.
Also in this case we can define in a suitable way a map, still denoted by $\omega$, which allows to translate the image through $\omega$ of the $\star$-product of two elements in the algebra with the pointwise product of the images of the two elements. Also in this case, we characterize the invertibility in the algebra. In this case, see Theorem \ref{thm4.3}, we  still have equivalent conditions, and one condition holds on one specific complex plane if and only if it holds for any complex plane. Unlike the discrete case, the notions of positivity and of factorization depends on the choice of an imaginary unit.\\
Our study shows that the Wiener algebra in the quaternionic case not only requires suitable tools but also shows features which do not appear in the complex case.

The paper consists of four sections besides the introduction. Sections 2 and 3 are devoted to the discrete case, respectively the study of the
discrete Wiener algebra and Toeplitz operators. Sections 4 and 5 deal with the continuous case. Specifically, in Section 4 we study the continuous Wiener algebra and in Section 5 we recall results on the quaternionic Hardy space of the half space and give an
application to Wiener-Hopf operators.
\section{The Wiener algebra}
\setcounter{equation}{0}

In this section we define the counterpart of the Wiener
algebra in the quaternionic setting.

\begin{e-definition}
We denote by $\mathcal W_{\mathbb H}$ the set of functions of the
form
\begin{equation}\label{effe}
f(p)=\sum_{u\in\mathbb Z}p^uf_u,
\end{equation}
where the $f_u$, $u \in \mathbb{Z}$, are quaternions such that
\[
\sum_{u\in\mathbb Z}|f_u|<\infty.
\]
\end{e-definition}
\begin{remark}
Quaternionic power series have been considered, though in a different context, namely in the framework of slice hyperholomorphic functions, in \cite{css} and \cite{gss}.
\end{remark}

The elements of $\mathcal W_{\mathbb H}$ are continuous functions on the
closed unit ball $\overline{ \mathbb{B} }$.  The set $\mathcal W_{\mathbb H}$ be endowed with
the multiplication (see \cite{fliess})
\[
(f\star g)(p)=\sum_{u\in\mathbb Z}p^u (\sum_{k\in\mathbb Z}f_{u-k}g_k).
\]
By extending the analogous result for polynomials with coefficients in a non commutative ring to power series (see \cite{lam})  we have the formula
\begin{equation}\label{pointwise}
(f\star g)(p)= f(p)g(f(p)^{-1}pf(p))
\end{equation}
which holds for $f(p)\not=0$ while for $f(p)=0$ we have $(f\star g)(p)=0$.
\begin{proposition}
$\mathcal W_{\mathbb H}$ endowed with the $\star$-multiplication is a real algebra.
\end{proposition}

\begin{proof}
The claim follows from
\[
\sum_{u\in\mathbb Z}|\sum_{k\in\mathbb Z}f_{u-k}g_k|\le
\sum_{u\in\mathbb Z}\sum_{k\in\mathbb Z}|f_{u-k}|\cdot|g_k|
=(\sum_{u\in\mathbb Z}|f_u|)(\sum_{k\in\mathbb Z}|g_k|).
\]
\end{proof}
We will call $\mathcal W_{\mathbb H}$ the quaternionic Wiener algebra.\smallskip

Let $i,j\in\mathbb S$ be orthogonal (that is, such that $ij+ji=0$). Then any quaternion $p$ can be written as $p=z+wj$,
where $z,w$ complex numbers belonging to the complex plane $\mathbb C_i= \{x+iy :  x,y\in\mathbb R \}$.
 In the sequel, we will sometimes write $\mathbb C$, for short. We have an injective homomorphism of rings $\chi=\chi_{i,j} :\, \mathbb H\to \mathbb C^{2\times 2}$ defined by
$$
\chi(p)=\begin{pmatrix} z & w\\
-\overline{w}& \overline{z}\end{pmatrix}.
$$
Let $f$ be as in \eqref{effe}. Then, restricting
$p=e^{it}$ to the unit circle, writing $f(e^{it})=a(e^{it})+b(e^{it})j$ we
obtain
\begin{equation}
\label{chi-f}
\chi(f(e^{it}))=\begin{pmatrix}a(e^{it})&b(e^{it})\\
-\overline{e^{it}}&\overline{a(e^{it})}\end{pmatrix},
\end{equation}
where the functions
\begin{equation}
\label{batsheva}
a(e^{it})=\sum_{u\in\mathbb Z}e^{iut}a_u\quad\text{and}\quad
b(e^{it})=\sum_{u\in\mathbb Z}e^{iut}b_u
\end{equation}
belong to the classical Wiener algebra, denoted by $\mathcal W  := \mathcal{W}^{1 \times 1}$. Thus,
\begin{equation}
\label{place-de-la-nation} \chi(f(e^{it}))\in\mathcal W^{2\times
2}.
\end{equation}
It is easy to verify that the map $\chi$ is not multiplicative with respect to  the $\star$-product, in fact
in general
\[
\chi((f\star g)(e^{i t} )) \not=\chi((f)(e^{i t})) \chi((g)(e^{i t})).
\]
We now introduce another map depending on the choice of two orthogonal imaginary units $i,j\in\mathbb S$ and denoted for simplicity by $\omega := \omega_{i,j}$,  which acts from the
values of functions (belonging to a given set of functions) to the set of $2\times 2$ matrices whose elements are complex-valued  functions when restricted to the complex plane $\mathbb C$ and it is defined by
\begin{equation}\label{omegaij}
(\omega(f))(z)=\begin{pmatrix}a(z)&b(z)\\ -\overline{b(\bar z)}& \overline{a(\bar z)}\end{pmatrix}, \quad z \in \overline{ \mathbb{B} \cap \mathbb{C} }.
\end{equation}
Note that the definition of $\omega$ makes sense if one substitutes $z$ with any $p\in\mathbb B$.
\begin{lemma} Let $f,g\in\mathcal W_{\mathbb H}$. Then it holds that
\begin{equation}
(\omega(f\star g))(z)=(\omega(f))(z)(\omega(g))(z), \quad z \in \overline{ \mathbb{B} \cap \mathbb{C} }. \label{Malakoff-ligne-13}
\end{equation}

\end{lemma}
\begin{proof}
Let $f(p)=p^na$ and $g(p)=p^mb$ for $n,m\in\mathbb Z$ and $a,b\in\mathbb H$. Then,
\[
(f\star g)(p)=p^{n+m}ab,
\]
and we have
\[
(\omega(f))(z^n)=z^n\chi(a),
\]
and similarly
\[
(\omega(g))(z)=z^m\chi(b),\quad
(\omega(f\star g))(z)=z^{n+m}\chi(ab).
\]
Hence
\[
(\omega(f\star g))(z)= z^{n+m}\chi(ab)=z^{n+m}\chi(a)\chi(b)=(\omega(f))(z)(\omega(g))(z).
\]
More generally, if $f(p)=\sum_{n\in\mathbb Z}p^na_n\in\mathcal W_{\mathbb H}$, then
\[
(\omega(f))(z)=\sum_{n\in\mathbb Z}z^{n}\chi(a_n)\in\mathcal W^{2\times 2}.
\]
Moreover, if $g(p)=\sum_{n\in\mathbb Z}p^n b_n\in\mathcal W_{\mathbb H}$, then
\[
\begin{split}
((\omega(f))(z))((\omega(g))(z))&=
\left(\sum_{n\in\mathbb Z}z^n\chi(a_n)\right)\left(\sum_{n\in\mathbb Z} z^n \chi(b_n)\right)\\
&=\sum_{u\in\mathbb Z} z^u \left(\sum_{n+m=u}\chi(a_n)\chi(b_m)\right)\\
&=\sum_{u\in\mathbb Z} z^u\chi\left(\sum_{n+m=u} a_nb_m \right)\\
&=(\omega(f\star g))(z).
\end{split}
\]
\end{proof}

\begin{theorem}
\label{wiener}
Let $f\in\mathcal W_{\mathbb H}$. The following are equivalent{\rm :}
\begin{enumerate}
\item[(i)] $f$ is invertible in $\mathcal W_{\mathbb H}${\rm ;}
\item[(ii)]
Let $i$ be any fixed element in $\mathbb S$, then
$(\det\omega(f))(z)\not=0$ for all $z\in \partial \mathbb B\cap \mathbb C_i${\rm ;}
\item[(iii)] The function $f$ does not vanish on $\partial\mathbb B$.
\end{enumerate}
\end{theorem}

\begin{proof}
The implication ${\rm (i)} \Longrightarrow {\rm (ii)}$ is a direct consequence of \eqref{Malakoff-ligne-13}: if
there exists $g\in\mathcal W_{\mathbb H}$ such that $f\star g=1$
then
\[
I_2=\begin{pmatrix}1&0\\0&1\end{pmatrix}=(\omega(f)(e^{it}))(\omega(g)(e^{it}))
\]
and so $\omega(f)(e^{it})$ is invertible for every $t\in[0,2\pi)$.\smallskip

To prove the converse we first note that $\omega(f)(z)\in\mathcal
W^{2\times 2}$. By the classical matricial Wiener-L\'evy theorem
the condition $\det\omega(f)(e^{it})\not=0$ for all $t\in[0,2\pi)$
implies that $\omega(f)$ is invertible in $\mathcal W^{2\times
2}$. Let $G\in\mathcal W^{2\times 2}$ be such that $\omega(f)G=I_2$.
By the formula for the inverse of a $2\times 2$ matrix we have
that $G$ is of the form
\[
G(e^{it})=\begin{pmatrix}c(e^{it})&d(e^{it})\\
-\overline{d(e^{-it})}&\overline{c(e^{-it})}\end{pmatrix},
\]
where $c(e^{-it}) = \overline{a(e^{-it})}/(\det\omega(f))(e^{it})$ and $d (e^{-it}) =-b(e^{it})/(\det\omega(f))(e^{it})$ so they belong to $\mathcal W$. Thus we can write
\[
c(e^{it})=\sum_{n\in\mathbb Z}e^{int}c_n\quad\text{and}\quad
d(e^{it})=\sum_{n\in\mathbb Z}e^{int}d_n.
\]
If we set
\[
g(e^{it})=\sum_{n\in\mathbb Z}e^{int}(c_n+d_nj),
\]
then $g\in\mathcal W_{\mathbb H}$. Since $\omega(f)\omega(g)=1$ on the unit circle, if $f(p)=\sum_{n\in\mathbb Z} p^nf_n$ and $g(p)=\sum_{n\in\mathbb Z} p^ng_n$, we have
\[
\chi\left(\sum_{n+m=u} f_ng_m\right)=\begin{cases}\, 0\,\,\,{\rm if}\,\,u\not=0,\,\\
                                          \, 1\,\,{\rm if}\,\, u=0.\end{cases}
\]
Hence $(f\star g)(e^{it})=1$ and thus $f\star g=1$ everywhere, since the latter is uniquely determined by its values on
the unit circle. \\

We now show that (ii) and (iii) are equivalent. An easy computation, using the notation in \eqref{omegaij}, shows that for some fixed $i\in\mathbb S$
 $$
 (\det \omega(f))(z)= a(z)\overline{a(\bar z)}+ b(z)\overline{b(\bar z)}=\sum_{u\in\mathbb Z} z^u (\sum_{k\in\mathbb Z}a_{u-k}\overline{a_k}+b_{u-k}\overline{b_k}),
$$
where the functions $a$ and $b$ are defined by \eqref{batsheva}.\smallskip

Given $f(p)=\sum_{u\in\mathbb Z} p^u f_u\in\mathcal W_{\mathbb H}$ we define the function $f^c(p)=\sum_{u\in\mathbb Z} p^u \overline{f_u}$ which still belongs to $\mathcal W_{\mathbb H}$.  The function $(f\star f^c)(p)= \sum_{u\in\mathbb Z} p^u (\sum_{k\in\mathbb Z}f_{u-k}\overline{f_k})$ has real coefficients, thus the zeros of its restriction  $(f\star f^c)_{|\mathbb C}(z)$ to $\mathbb C$ has zeros which are real points and/or complex conjugate points. By the validity of the representation formula for power series, see \cite{css}, it follows that $(f\star f^c)(p)$ has zero set consisting of real points and/or spheres.
 By formula (\ref{pointwise}) we have that  $(f\star f^c)(p_0)=0$ implies $f(p_0)=0$ or $f(p_0)\not=0$ and  $f^c (f(p_0)^{-1}p_0f(p_0))=0$. Note that the element $f(p_0)^{-1}p_0f(p_0)$ belongs to $[p_0]$. However the zeros of $f^c$ on $[p_0]$ are in one-to-one correspondence with the zeros of $f$ on $[p_0]$, in fact
  $$
  f(x+iy)=\sum_{u\in\mathbb Z}(x+iy)^u f_u= \sum_{u\in\mathbb Z}(s_u(x,y)+it_u(x,y)) f_u=s(x,y)+it(x,y)
  $$
  where the functions $s_u$, $t_u$ are real valued.  Similarly,
\[
f^c(x+iy)=\sum_{u\in\mathbb Z}(x+iy)^u \bar f_u= \sum_{u\in\mathbb Z}(s_u(x,y)+it_u(x,y)) \bar f_u=\overline{s(x,y)}+i\overline{t(x,y)}.
\]
If $f$ vanishes at all points of $[p_0]$ then $s(x_0,y_0)=t(x_0,y_0)=0$ and $f^c$ vanishes at all points of $[p_0]$. If $p_0$ is the only zero of $f$ on $[p_0]$ then $t(x_0,y_0)\not=0$ and an immediate computation shows that
  $f^c(x_0+\tilde{\iota} y_0)=0$ where $\tilde{\iota} =- \overline{t(x_0,y_0)} i\overline{t(x_0,y_0)}\,^
  {-1}$. The converse follows in an analogous way, since $(f^c)^c=f$.
  Thus if $f^c$ has a zero belonging to the sphere $[p_0]$ also $f$ must have a zero belonging to the same sphere.
 Now we observe that since $f_u=a_u+b_uj$ and $\overline{f_u}=\overline{a_u}-b_uj$ it is immediate  that
 \[
 \begin{split}
(f\star f^c)_{|\mathbb C}(z)&=\sum_{u\in\mathbb Z}p^u (\sum_{k\in\mathbb Z}f_{u-k}\overline{f_k})=\sum_{u\in\mathbb Z}p^u (\sum_{k\in\mathbb Z}(a_{u-k}+b_{u-k}j)(\bar a_{k}-b_{k}j))\\
&=\sum_{u\in\mathbb Z}p^u \{ \sum_{k\in\mathbb Z}(a_{u-k}\bar a_k+b_{u-k}\bar b_k+(b_{u-k} a_{k}-a_{u-k}b_{k})j)\}
\end{split}
\]
 and so $(\det \omega(f))(z)=(f\star f^c)_{|\mathbb C}(z)$.\\
If (ii) holds then $f$ cannot have zeros on $\partial \mathbb B$ otherwise, if $f(p_0)=0$ we have that $f\star f^c$ vanishes on $[p_0]$ and in particular on $[p_0]\cap\mathbb C$ and so $\det(\omega(f))$ vanishes on $\partial \mathbb  B\cap\mathbb C$. Conversely, if $f$ does not vanish on $\partial \mathbb B$ neither $f^c$ vanishes there and so $f\star f^c$ does not have zeros on $\partial \mathbb B$ and thus $\det(\omega(f))$ does not vanish on $\partial \mathbb B\cap\mathbb C$.   \end{proof}
\begin{remark}
Note that condition (ii) holds on a fixed plane while condition (iii) refers to the whole boundary of  $\mathbb B$.
Consider, for example, $f(p)=p-j$. It is immediate that $f$ does not have any zero on the complex plane $\mathbb C_i$. However,
$\det(\omega(f))(e^{\pm i\pi})=0$.
\end{remark}
\begin{e-definition}
We denote by $\mathcal W_{\mathbb H,+}$ (resp. $\mathcal W_{\mathbb H,-}$) the set of elements $f(p)=\sum_{n\in\mathbb N}
p^nf_n\,\in\,\mathcal W_{\mathbb H,}$ for which $f_n=0$ for $n<0$ (resp.  for $n>0$).
\end{e-definition}

It is immediate that $\mathcal W_{\mathbb H,+}$ and $\mathcal W_{\mathbb H,-}$ are subalgebras of $\mathcal W_{\mathbb H}$.
The question of invertibility of an element in $\mathcal W_{\mathbb H,+}$ can be asked in $\mathcal W_{\mathbb H,+}$ and in
$\mathcal W_{\mathbb H}$. For instance, $f(p)=p$ is invertible in $\mathcal W_{\mathbb H}$, but not in $\mathcal W_{\mathbb H,+}$.

\begin{theorem}
Let $f\in\mathcal W_{\mathbb H,+}$. The following are equivalent{\rm :}
\begin{enumerate}
\item[(i)] The function $f$ is invertible in $\mathcal W_{\mathbb H,+}${\rm ;}
\item[(ii)]
Let $i$ be any fixed element in $\mathbb S$, then
$(\det\omega(f))(z)\not=0$ for all $z\in \overline{\mathbb B\cap\mathbb C_i}${\rm ;}
\item[(iii)] The function $f$ does not vanish on $\overline{\mathbb B}$.
\end{enumerate}
\end{theorem}

\begin{proof}
The proof is very similar to the proof of Theorem \ref{wiener}. We will only provide details for the implication ${\rm (ii)} \Longrightarrow {\rm (i)}$. We begin by noting that $\omega(f) \in \mathcal{W}_+^{2 \times 2}$. By the result for invertiblity of matrix-valued
functions in the Wiener algebra $\mathcal{W}_+^{2 \times 2}$ (see, e.g., \cite{ggk}), the condition $\det ( \omega(f) )(z) \neq 0$ for $z \in \overline{ \mathbb{B} \cap \mathbb{C} }$ implies that $\omega(f)$ is invertible in $\mathcal{W}_+^{2 \times 2}$. Let $G \in \mathcal{W}_+^{2 \times 2}$ be so that $\omega(f) G = I_2$ on $\overline{\mathbb{B} \cap \mathbb{C} }$. Similar to the proof of Theorem \ref{wiener} we get that
$$G(z) = \begin{pmatrix} c(z) & d(z) \\ -\overline{ d(\bar{z}) } & \overline{ c(z) } \end{pmatrix},$$
where $c(z) = \overline{ a(\bar{z}) } \det (\omega(f))(z)$ and $d(z) = -b(z) \det (\omega(f))(z)$. Thus, we may write $c(z) = \sum_{n=0}^{\infty} z^n c_n$ and $d(z) = \sum_{n=0}^{\infty} z^n d_n$. If we set $g(z) = \sum_{n=0}^{\infty} z^n (c_n + d_n j)$, then $g \in \mathcal{W}_{\mathbb{H}, +}$. The rest of the argument is as in the proof of Theorem \ref{wiener}.
\end{proof}

\begin{e-definition}
We say that $f\in\mathcal W_{\mathbb H}$ is strictly positive if
\[
\omega(f)(e^{it})>0,\quad  t\in[0,2\pi).
\]
\end{e-definition}

\begin{theorem}
$f\in\mathcal W_{\mathbb H}$ is strictly positive if and only if it can be written as $f=f_+\star f_+^c$, where $f_+$ is an invertible element of $\mathcal W_{\mathbb H,+}$.
\end{theorem}

\begin{proof}
By the classical Wiener-Hopf theory, there is an element
$A\in\mathcal W_+^{2\times 2}$, unique up to a right
multiplicative unitary constant, such that $A^{-1}\in\mathcal
W_+^{2\times 2}$ and
\[
\omega(f)(e^{it})=A(e^{it})A(e^{it})^*.
\]
Let
\begin{equation}
\label{elizabeth}
J_1=\begin{pmatrix}0&1\\-1&0\end{pmatrix}.
\end{equation}
A simple computation shows that
\[
\begin{pmatrix}0&1\\-1&0\end{pmatrix}
\begin{pmatrix}\overline{a(e^{-it}})& \overline{b(e^{-it})}\\-b(e^{it})&a(e^{it})\end{pmatrix}
\begin{pmatrix}0&-1\\1&0\end{pmatrix}=
\begin{pmatrix} a(e^{it}& b(e^{it})\\-\overline{b(e^{-it})}& \overline{a(e^{-it})}\end{pmatrix}
\]
from which we deduce
\[
J_1\overline{\omega(f)(e^{-it})}J_1^*=\omega(f)(e^{it}).
\]
The function
\[
B(e^{it})=J_1\overline{A(e^{-it})}J_1^*
\]
is an invertible element of $\mathcal W_+^{2\times 2}$.
Thus, for a unitary constant $U$ we have
\begin{equation}
\label{clothilde}
A(e^{it})=J_1\overline{A(e^{-it})}J_1^*U.
\end{equation}
We have that $A(1)$ is invertible since $A(e^{it})$ is invertible in
$\mathcal W_+^{2\times 2}$. By replacing $A(z)$ by $A(z)A(1)^*(A(1)A(1)^*)^{-1/2}$ we can always choose $A(1)>0$, i.e., $A(1)$ is a positive definite matrix.
This forces $U=I_2$. Indeed, let
$A(1)=\begin{pmatrix}a&\overline{c}\\c&d\end{pmatrix}$. Then
\[
A(1)=J_1\overline{A(1)}J_1^*
\]
and \eqref{clothilde} leads to $U=I_2$.
With $A(z)=\sum_{n=0}^\infty z^nA_n$ we have
\[
A_n=J_1\overline{A_n}J_1^*,
\]
and so
\[
A(e^{it})=\omega(a_+)(e^{it})
\]
where $a_+(p)=\sum_{n=0}^\infty p^na_n$ and $\chi(a_n)=A_n$.
\end{proof}

\section{Toeplitz operators}
\setcounter{equation}{0}
Let $\ell_2(\mathbb{H})$ denote the Hilbert space of all $\mathbb{H}$-valued sequences $a = (a_n)_{n=0}^{\infty}$ so that
$$\| a \|_{\ell_2}^2 = \sum_{n=0}^{\infty} | a_n |^2 < \infty.$$
The inner product for $a, b \in \ell_2(\mathbb{H})$ is given by
$$\langle a, b \rangle_{\ell_2} = \sum_{n=0}^{\infty} \bar{b}_n a_n.$$
Let $L_2(\partial \mathbb{B})$ denote the Hilbert space of all functions $f(p) = \sum_{n=-\infty}^{\infty} p^n f_n$ on $\partial \mathbb{B}$ so that
\begin{align}
\| f \|^2_{ L_2 (\partial \mathbb{B} ) } =& \; \frac{1}{2\pi} \int_0^{2\pi} |f(e^{i \theta}) |^2 d\theta \label{first} \\
=& \; \sum_{n=-\infty}^{\infty} | f_n |^2 < \infty.  \label{independent}
\end{align}
As in the complex case, elements of $L_2(\partial \mathbb{B})$ are defined almost everywhere on $\partial\mathbb B$.
It is worth noting that \eqref{independent} does not depend on $i \in \mathbb{S}$. Thus, any $i \in \mathbb{S}$ can be used in \eqref{first}. The inner product for $L_2(\partial \mathbb{B})$ is given by
$$
\langle f, g \rangle_{\partial \mathbb{B} } = \frac{1}{2\pi} \int_0^{2\pi} \overline{g(e^{i \theta})} f(e^{i \theta}) d\theta = \sum_{u =-\infty}^{\infty} \bar{g}_{u} f_u.
$$
We let $H_2( \partial \mathbb{B} )$ and $H_2( \partial \mathbb{B})^{\perp}$ denote the set of functions $f \in L_2(\partial \mathbb{B})$ so that $f_n = 0$ for $n < 0$ and $f_n = 0$ for $n \geq 0$, respectively.

Let $\mathfrak{p}$  and $\mathfrak{q}$ denote the orthogonal projections of $L_2(\partial \mathbb B)$ onto $H_2(\partial \mathbb{B})$ and $L_2(\partial \mathbb B)$ onto $( H_2(\partial \mathbb{B}))^{\perp}$, respectively. Note that $(\mathfrak{p} + \mathfrak{q})f = f$ for $f \in L_2(\partial \mathbb B)$. The operators $\mathfrak{p}$ and $\mathfrak{q}$ are given by the rules
$$(\mathfrak{p} f)(p) = \sum_{n=0}^{\infty} p^n f_n$$
and
$$(\mathfrak{q} f)(p) = \sum_{n=-\infty}^{-1} p^n f_n,$$
respectively,
where $f(p) = \sum_{n=-\infty}^{\infty} p^n f_n$. If $\varphi \in \mathcal{W}_{\mathbb{H}}$, then we let $M_{\varphi}$ denote the operator of multiplication on $\mathcal{W}_{\mathbb H}$ given by $(M_{\varphi} f)(p) = (\varphi \star f)(p)$. If $\varphi \in \mathcal
W_{\mathbb H}$, then we define the Toeplitz operator $T_{\varphi}$ by the rule
\begin{equation}
\label{eq:Apr27ub1}
T_{\varphi} f = \mathfrak{p} M_{\varphi} f \quad {\rm for} \quad f \in H_2( \partial \mathbb B).
\end{equation}
The Toeplitz operator $T_{\varphi}$ can also be defined as an operator from $\ell_2(\mathbb{H}) \to \ell_2(\mathbb{H})$ given by
\begin{equation}
\label{eq:Jul2u1}
T_{\varphi} \xi = \begin{pmatrix} \varphi_0 & \cdots & \varphi_{-n} & \cdots \\ \vdots & \ddots & \vdots & \ddots \\
\varphi_n & \cdots & \varphi_0 & \cdots \\ \vdots & \ddots & \vdots & \ddots \end{pmatrix} \begin{pmatrix} \xi_0 \\ \vdots \\ \xi_n \\ \vdots \end{pmatrix} \quad {\rm for} \quad \xi = \begin{pmatrix} \xi_0 \\ \vdots \\ \xi_n \\ \vdots \end{pmatrix} \in \ell_2(\mathbb{H}).
\end{equation}

Theorem \ref{theorem:Apr27uz1} and Corollary \ref{corcor} are natural analogs of results established for Toeplitz operators in the classical setting (see \cite{brownhalmos}).

\begin{theorem}
\label{theorem:Apr27uz1}
If $\varphi,\psi \in \mathcal{W}_{ \mathbb H}$, then $T_{\varphi} T_{\psi}$ is Toeplitz if and only if $\varphi \in
\mathcal{W}_{\mathbb H,-}$ or $\psi \in \mathcal{W}_{\mathbb H,+}$. In this case, $T_{\varphi} T_{\psi} = T_{ \varphi \star \psi}$.
\end{theorem}

\begin{proof}
If $\psi \in \mathcal{W}_{\mathbb H,+}$, then
\begin{align*}
T_{\varphi} T_{\psi} f =& \; \mathfrak{p} M_{\varphi} \mathfrak{p} M_{\psi} f \quad {\rm for} \quad f \in H_2(\partial \mathbb B) \\
=& \; \mathfrak{p} M_{\varphi} M_{\psi} f \\
=& \; \mathfrak{p} M_{\varphi \star \psi} f \\
=& \; T_{\varphi \star \psi} f.
\end{align*}

If $\varphi \in \mathcal{W}_{\mathbb H,-}$, then
\begin{align*}
T_{\varphi} T_{\psi} f =& \; \mathfrak{p} M_{\varphi} \mathfrak{p} M_{\psi} f \quad {\rm for} \quad f \in H_2(\partial \mathbb B) \\
=& \; \mathfrak{p} M_{\varphi} (\mathfrak{p} + \mathfrak{q}) M_{\psi} f - \mathfrak{p} M_{\varphi} \mathfrak{q} M_{\psi} f \\
=& \; \mathfrak{p} M_{\varphi \star \psi} f - 0 \\
=& \; T_{\varphi \star \psi} f.
\end{align*}

Conversely, if $T_{\varphi} T_{\psi}$ is Toeplitz, then
\begin{equation}
\label{eq:Apr27nb1}
T_{\varphi} T_{\psi} = T_{\varphi \star \psi}.
\end{equation}
Indeed, if there exists $g \in \mathcal{W}_{\mathbb H}$ so that
$$\mathfrak{p} M_{\varphi} \mathfrak{p} M_{\psi} f = \mathfrak{p} M_g f \quad {\rm for} \quad f \in H_2(\partial \mathbb B),$$
then
$$\mathfrak{p} M_{\varphi} (\mathfrak{p} + \mathfrak{q} - \mathfrak{q}) M_{\psi} f = \mathfrak{p} M_g f,$$
whence
$$\mathfrak{p} M_{\varphi \star \psi - g } f = \mathfrak{p} M_{\varphi} \mathfrak{q} M_{\psi} f \quad {\rm for} \quad f \in H_2(\partial \mathbb B).$$
Thus $g = \varphi \star \psi$.
Finally, it follows from \eqref{eq:Apr27nb1} that
$$ \mathfrak{p} M_{\varphi \star \psi} f - \mathfrak{p} M_{\varphi} \mathfrak{q} M_{\psi} f = \mathfrak{p} M_{\varphi \star \psi} f \quad {\rm for} \quad f \in H_2(\partial \mathbb B)$$
i.e.
\begin{equation}
\label{eq:Apr27ur1}
\mathfrak{p} M_{\varphi} \mathfrak{q} M_{\psi} f = 0 \quad {\rm for} \quad f \in H_2(\partial \mathbb B).
\end{equation}
It follows from \eqref{eq:Apr27ur1} that $\varphi \in \mathcal{W}_{\mathbb H,-}$ or $\psi \in \mathcal{W}_{\mathbb H,+}$.
\end{proof}

\begin{corollary}
\label{corcor}
If $\varphi, \psi \in \mathcal{W}_{\mathbb{H}}$, then $T_{\varphi} T_{\psi} = 0$ if and only if $\varphi(p) = 0$ for all $p \in \partial \mathbb{B}$ or $\phi(p) = 0$ for all $p \in \partial \mathbb{B}$.
\end{corollary}

\begin{proof}
The converse statement is obvious. If $T_{\varphi} T_{\psi} = 0 = T_0$, then we may use \eqref{theorem:Apr27uz1} to deduce that $(\varphi \star \psi)(p) = 0$ for all $p \in \partial \mathbb{B}$. But then we must have that $\varphi(p) = 0$ for all $p \in \partial \mathbb{B}$ or $\psi(p) = 0$ for all $p \in \partial \mathbb{B}$.

\end{proof}

\section{The Wiener algebra in the continuous case}
\setcounter{equation}{0}
Let $L_1(\mathbb{R}, \mathbb{H})$ denote the space of functions $f: \mathbb{R} \to \mathbb{H}$ so that
$$\| f \|_{L_1(\mathbb{R}, \mathbb{H})} = \int_{\mathbb{R}} | f(u) | du < \infty.$$
We denote by $\mathcal W(\mathbb R, \mathbb H)$ the set of functions of the form
\begin{equation}\label{Effe}
F_1(p)=c_1+\int_{\mathbb R}e^{pu}f_1(u)\, du, \qquad p\in\mathbb R\mathbb S=\left\{it: t\in\mathbb R\,\,\text{and}\,\, i\in\mathbb S
\right\},
\end{equation}
where $f_1\in L_1(\mathbb R,\mathbb H)$ and $c_1\in\mathbb H$.
The set $\mathcal W (\mathbb R, \mathbb H)$ can be endowed with the multiplication
$$
(F_1\star F_2)(p)=c_1c_2+\int_{\mathbb R}e^{up}c_1f_2(u)\, du + \int_{\mathbb R}e^{up}f_1(u)c_2\, du
+\int_{\mathbb R}e^{up} (f_1\circ f_2)(u)\, du
$$
where $f_1\circ f_2$ denotes the convolution of $f_1$ and $f_2$ and
$$
F_2(p)=c_2+\int_{\mathbb R}e^{up}f_2(u)\, du.
$$
If $F_1$ is as above, we also define
$$
\| F_1\|= |c_1|+\|f_1\|_{L_1(\mathbb{R}, \mathbb{H})}.
$$
We have:
\begin{proposition}
$\mathcal W(\mathbb R, \mathbb H)$ endowed with the $\star$-multiplication is a real Banach algebra.
\end{proposition}
\begin{proof}
The proof follows from
\[
\begin{split}
\| F_1\star F_2\| &\leq  |c_1c_2|+ \| f_1c_2+c_1f_2+f_1\circ f_2\|_{L_1(\mathbb{R}, \mathbb{H})} \\
&\leq |c_1c_2|+ \| f_1\|_{L_1(\mathbb{R}, \mathbb{H})} |c_2|+|c_1| \| f_2\|_{L_1}+\|f_1\circ f_2\|_{L_1(\mathbb{R}, \mathbb{H})}\\
&=(|c_1| + \| f_1\|_{L_1(\mathbb{R}, \mathbb{H})}) (|c_2| + \| f_2\|_{L_1(\mathbb{R}, \mathbb{H})} )=\| F_1\|+\| F_2\|.
\end{split}
\]
The completeness follows just as in the complex case.
\end{proof}
We now define a map $\omega=\omega_{i,j}$ depending on the choice of an imaginary unit $i\in\mathbb S$ and of a $j\in\mathbb S$
orthogonal to $i$. In what follows, we will omit to specify the imaginary units $i$ and $j$, but we keep the imaginary unit while writing the variables to make it explicit that we work on a specific complex plane (this leads to a different notation than in the classical case, where
the imaginary unit is unique and can be understood).
Then, by setting $p=it$, we define
\[
\omega(F_1)(it)=\chi(c_1)+\int_{\mathbb R} e^{itu}\, \chi(f_1(u))\, du,
\]
where $\chi$ is defined as in \eqref{chi-f}, that is if $f_1(u)=a(u)+b(u)j$ then
$$
\chi(f_1(u))=\begin{pmatrix}a(u)&b(u)\\
-\overline{b(u)}&\overline{a(u)}\end{pmatrix}.
$$
It is immediate that $a,b\in L_1(\mathbb R,\mathbb C)$ and so $\omega(F_1)(it)$ belongs to the continuous Wiener algebra $\mathcal W^{2\times 2}(\mathbb R)$ with values in $\mathbb C^{2\times 2}$.
\begin{lemma}\label{lemmastar}
Let $F_1,F_2\in\mathcal W (\mathbb R, \mathbb H)$. Then
$$
(\omega(F_1\star F_2))(it)=(\omega(F_1))(it)(\omega(F_2))(it), \quad t \in \mathbb{R}.
$$
\end{lemma}
\begin{proof}
We have
\[
\begin{split}
(\omega(F_1\star F_2))(it)&=\chi(c_1c_2)+\int_{\mathbb R} e^{itu}\chi(f_1c_2+c_1f_2+f_1\circ f_2)(u)\, du\\
&=\chi(c_1c_2)+\int_{\mathbb R} e^{itu}(\chi(f_1c_2)+\chi(c_1f_2)+\chi(f_1\circ f_2))(u)\, du\\
&=\chi(c_1c_2)+\int_{\mathbb R} e^{itu}(\chi(f_1)\chi(c_2)+\chi(c_1)\chi(f_2)+\chi(f_1)\circ \chi(f_2))(u)\, du.\\
\end{split}
\]
Moreover,
\[
\begin{split}
\omega(F_1)(it)&\omega(F_2)(it)=\left(\chi(c_1)+\int_{\mathbb R} e^{itu}\, \chi(f_1(u))\, du\right)
\left(\chi(c_2)+\int_{\mathbb R} e^{itu}\, \chi(f_2(u))\, du\right)\\
&=\chi(c_1c_2)+\int_{\mathbb R} e^{itu}(\chi(f_1)\chi(c_2)+\chi(c_1)\chi(f_2)+\chi(f_1)\circ \chi(f_2))(u)\, du,\\
\end{split}
\]
where the last equality follows from the case of functions in $\mathcal W^{2\times 2}(\mathbb R)$.
\end{proof}
Our main result as follows:
\begin{theorem}\label{thm4.3}
Let $F\in\mathcal W(\mathbb R, \mathbb H)$. The following are equivalent{\rm :}
\begin{enumerate}
\item[(i)] The function $F$ is invertible in $\mathcal W(\mathbb R ,\mathbb H)${\rm ;}
\item[(ii)]  For any $i\in\mathbb S$ we have $\det \omega(F)(it)\not=0$, for all $t\in \mathbb R${\rm ;}
\item[(iii)] There exists an $i\in\mathbb S$ such that $\det \omega(F)(it)\not=0$, for all $t\in \mathbb R${\rm ;}
\item[(iv)] The function $F(p)\not=0$ for all $p\in\mathbb R\mathbb S$.
\end{enumerate}
\label{theorem43}
\end{theorem}
\begin{proof}
Assume that (i) holds, and denote by $F^{-1}$ the inverse of $F$ in $\mathcal W(\mathbb R, \mathbb H)$. Since $(F\star F^{-1})(p)=1$, $p\in\mathbb S\mathbb R$,  Lemma \ref{lemmastar} gives $\omega(F)({it}) \omega(F^{-1})({it})=I_2$ and so
$\det\omega (F)({it})\not=0$ and (ii) holds.
To show the converse, let us assume that $\det \omega(F)(it)\not=0$, for all $t\in\mathbb R$, for some fixed $i\in\mathbb S$. Thus the function $\omega(F)(it)$ admits an inverse in $\mathcal W^{2\times 2}(\mathbb R)$,  where $F$ is as in formula (\ref{Effe}), given by
$$
\omega(F)^{-1}(it)=M + \int_{\mathbb R}e^{itu} \begin{pmatrix} c(u)& d(u)\\
e(u) & f(u)\end{pmatrix} \, du.
$$
Since $c_1\not=0$, and so $\chi(c_1)$ is invertible, we have $M=\chi(c_1^{-1})$. It is not reductive to assume that $M=I_2$, (in the general case the computations below are slightly longer but they are essentially the same).
By considering the product $\omega(F)\omega(F)^{-1}$, we obtain
$$
I_2=\omega(F)(it)\omega(F)^{-1}(it)=I_2+\int_{\mathbb R} e^{itu}\left(\int_{\mathbb R}
\begin{pmatrix} a(u-v)& b(u-v)\\
-\overline{b(u-v)} & \overline{a(u-v)}\end{pmatrix}\begin{pmatrix} c(v)& d(v)\\
e(v) & f(v)\end{pmatrix} dv + \right.
$$
$$
\left.
\begin{pmatrix} c(u)& d(u)\\
e(u) & f(u)\end{pmatrix}+\begin{pmatrix} a(u)& b(u)\\
-\overline{b(u)} & \overline{a(u)}\end{pmatrix}\right) du
$$
and so
$$
\int_{\mathbb R} e^{itu}\left\{\int_{\mathbb R}
\begin{pmatrix} a(u-v)& b(u-v)\\
-\overline{b(u-v)} & \overline{a(u-v)}\end{pmatrix}\begin{pmatrix} c(v)& d(v)\\
e(v) & f(v)\end{pmatrix} dv + \right.
$$
$$
\left.
\begin{pmatrix} c(u)& d(u)\\
e(u) & f(u)\end{pmatrix}+\begin{pmatrix} a(u)& b(u)\\
-\overline{b(u)} & \overline{a(u)}\end{pmatrix}\right\}du =0_2,
$$
which gives
\begin{equation}\label{eqomega}
\int_{\mathbb R}
\begin{pmatrix} a(u-v)& b(u-v)\\
-\overline{b(u-v)} & \overline{a(u-v)}\end{pmatrix}\begin{pmatrix} c(v)& d(v)\\
e(v) & f(v)\end{pmatrix} dv +
\begin{pmatrix} c(u)& d(u)\\
e(u) & f(u)\end{pmatrix}+\begin{pmatrix} a(u)& b(u)\\
-\overline{b(u)} & \overline{a(u)}\end{pmatrix} =0_2.
\end{equation}
With $J_1$ given by \eqref{elizabeth},
let us take the conjugate of equation \eqref{eqomega} and multiply in a suitable way by $J_1$ and $J_1^T$ to obtain
$$
\int_{\mathbb R}
J_1\begin{pmatrix} \overline{a(u-v)}& \overline{b(u-v)}\\
-{b(u-v)} & {a(u-v)}\end{pmatrix}J_1^TJ_1\begin{pmatrix} \overline{c(v)}& \overline{d(v)}\\
\overline{e(v)} & \overline{f(v)}\end{pmatrix}J_1^T dv
$$
$$
+J_1\begin{pmatrix} \overline{c(v)}& \overline{d(v)}\\
\overline{e(v)} & \overline{f(v)}\end{pmatrix}J_1^T+J_1\begin{pmatrix} \overline{a(u)}& \overline{b(u)}\\
-{b(u)} & {a(u)}\end{pmatrix}J_1^T =0_2
$$
which gives
$$
\int_{\mathbb R}
\begin{pmatrix} {a(u-v)}& {b(u-v)}\\
-\overline{b(u-v)} & \overline{a(u-v)}\end{pmatrix}J_1\begin{pmatrix} \overline{c(v)}& \overline{d(v)}\\
\overline{e(v)} & \overline{f(v)}\end{pmatrix}J_1^T dv
$$
$$
+J_1\begin{pmatrix} \overline{c(v)}& \overline{d(v)}\\
\overline{e(v)} & \overline{f(v)}\end{pmatrix}J_1^T+\begin{pmatrix} {a(u)}& {b(u)}\\
-\overline{b(u)} & \overline{a(u)}\end{pmatrix} =0_2.
$$
Since equation \eqref{eqomega} admits a unique solution we deduce that
$$
J_1\begin{pmatrix} \overline{c(v)}& \overline{d(v)}\\
\overline{e(v)} & \overline{f(v)}\end{pmatrix}J_1^T =\begin{pmatrix} {c(v)}& {d(v)}\\
{e(v)} & {f(v)}\end{pmatrix}
$$
from which it follows
$e(u)=-\overline{d(u)}$ and $f(u)=\overline{c(u)}$. Thus $\omega(F)^{-1}=\omega(G)$ with
$$
G(it)=c_1^{-1}+\int_{\mathbb R}e^{itu} (c(u)+d(u)j)\, du,
$$
and
$G\in\mathcal W(\mathbb R,\mathbb H)$ since $c,d \in L_1(\mathbb R,\mathbb H)$. The function $G$ is the inverse of $F$ in $\mathcal W(\mathbb R,\mathbb H)$.\\

It is trivial to show that (ii) implies (iii). We now show that (iii) implies (iv). Let us fix $i$ and rewrite $\omega(F)(it)$ in the form
$$
\omega(F)(it)= \begin{pmatrix} A(it)& B(it)\\
-\overline{B(-it)} & \overline{A(-it)}\end{pmatrix},
$$
where
$$
A(it)=\alpha + \int_{\mathbb R}e^{itu} a(u)\, du, \qquad B(it)=\beta + \int_{\mathbb R}e^{itu} b(u)\, du,
$$
where $c_1=\alpha+\beta j$.\smallskip

Suppose to the contrary that there exists $t_0\in\mathbb R$ such that $F(it_0)=0$. Then $A(it_0)=B(it_0)=0$ and so
we have $\det(\omega(F))(it_0)=0$ a contradiction.\smallskip

Let us now assume (iv), and prove that (i) is in force. Since $F(it)$ is invertible for all $t\in\mathbb R$ and all $i\in\mathbb S$, the matrix $\chi(F)(it)$ is invertible for every $i \in \mathbb{S}$ and $t \in \mathbb{R}$ and so the function $\chi(F)$ is invertible in $\mathcal W^{2\times 2}(\mathbb R)$
for every $i\in\mathbb S$.
Let us write
\[
\begin{split}
(\chi(F))(it)&= \chi(c_1)+
\begin{pmatrix} \int_{\mathbb R} e^{itu} a(u)\, du& \int_{\mathbb R} e^{itu} b(u)\, du\\
-\int_{\mathbb R} e^{-itu} \overline{b(u)}\, du & \int_{\mathbb R} e^{-itu} \overline{a(u)}\, du\end{pmatrix}
\\
&=\chi(c_1)+
\begin{pmatrix} \int_{\mathbb R} e^{itu} a(u)\, du& \int_{\mathbb R} e^{itu} b(u)\, du\\
-\int_{\mathbb R} e^{itu} \overline{b(-u)}\, du & \int_{\mathbb R} e^{itu} \overline{a(-u)}\, du\end{pmatrix}
\end{split}
\]
and since $\chi(F)$ is invertible in $\mathcal W^{2\times 2}(\mathbb R)$, there exists $H$ such that $(\chi(F)(it))H(it)=I_2$.
Let us write
$$
H(it)=\chi(c_1^{-1})+\begin{pmatrix} C(it)& D(it)\\
E(it) & L(it)\end{pmatrix},
$$
and, up to a normalization we assume that $\chi(c_1)=\chi(c_1^{-1})=I_2$.
Then we have
\begin{equation}\label{inverse}
\begin{split}
\chi(F)(it)H(it)&=I_2 \\
& = \left(I_2+
\begin{pmatrix} \int_{\mathbb R} e^{itu} a(u)\, du& \int_{\mathbb R} e^{itu} b(u)\, du\\
-\int_{\mathbb R} e^{itu} \overline{b(-u)}\, du & \int_{\mathbb R} e^{itu} \overline{a(-u)}\, du\end{pmatrix}
\right) \left(I_2+\begin{pmatrix} C(it)& D(it)\\
E(it) & L(it)\end{pmatrix}\right)\\
& = I_2+
\begin{pmatrix} \int_{\mathbb R} e^{itu} a(u)\, du& \int_{\mathbb R} e^{itu} b(u)\, du\\
-\int_{\mathbb R} e^{itu} \overline{b(-u)}\, du & \int_{\mathbb R} e^{itu} \overline{a(-u)}\, du\end{pmatrix}
\begin{pmatrix} C(it)& D(it)\\
E(it) & L(it)\end{pmatrix}\\
&+
\begin{pmatrix} \int_{\mathbb R} e^{itu} a(u)\, du& \int_{\mathbb R} e^{itu} b(u)\, du\\
-\int_{\mathbb R} e^{itu} \overline{b(-u)}\, du & \int_{\mathbb R} e^{itu} \overline{a(-u)}\, du\end{pmatrix}
+\begin{pmatrix} C(it)& D(it)\\
E(it) & L(it)\end{pmatrix}
\\
\end{split}
\end{equation}
is uniquely solvable.
Reasoning as in the proof of (ii) implies (i), we can take the conjugate of \eqref{inverse} and suitably multiplying by $J_1$ and $J_1^T$ we obtain:
\[
\begin{split}
I_2&= I_2 +J_1
\begin{pmatrix} \int_{\mathbb R} e^{-itu} \overline{a(u)}\, du& \int_{\mathbb R} e^{-itu}\overline{ b(u)}\, du\\
-\int_{\mathbb R} e^{-itu} {b(-u)}\, du & \int_{\mathbb R} e^{-itu} {a(-u)}\, du\end{pmatrix}
J_1^T J_1 \begin{pmatrix} \overline{C(it)}& \overline{D(it)}\\
\overline{E(it)} & \overline{L(it)}\end{pmatrix}J_1^T\\
&+J_1
\begin{pmatrix} \int_{\mathbb R} e^{-itu} \overline{a(u)}\, du& \int_{\mathbb R} e^{-itu}\overline{ b(u)}\, du\\
-\int_{\mathbb R} e^{-itu} {b(-u)}\, du & \int_{\mathbb R} e^{-itu} {a(-u)}\, du\end{pmatrix}
J_1^T + J_1 \begin{pmatrix} \overline{C(it)}& \overline{D(it)}\\
\overline{E(it)} & \overline{L(it)}\end{pmatrix} J_1^T
\\
&=I_2+
\begin{pmatrix} \int_{\mathbb R} e^{-itu} {a(-u)}\, du& \int_{\mathbb R} e^{-itu}{b(-u)}\, du\\
-\int_{\mathbb R} e^{-itu} \overline{ b(u)} \, du & \int_{\mathbb R} e^{-itu} \overline{a(u)}\, du\end{pmatrix}
  \begin{pmatrix} \overline{L(it)}& -\overline{E(it)}\\
-\overline{D(it)} & \overline{C(it)}\end{pmatrix}
\\
&+
\begin{pmatrix} \int_{\mathbb R} e^{-itu} {a(-u)}\, du& \int_{\mathbb R} e^{-itu}{b(-u)}\, du\\
-\int_{\mathbb R} e^{-itu} \overline{ b(u)} \, du & \int_{\mathbb R} e^{-itu} \overline{a(u)}\, du\end{pmatrix}
 + \begin{pmatrix} \overline{L(it)}& -\overline{E(it)}\\
-\overline{D(it)} & \overline{C(it)}\end{pmatrix}
\\
&=I_2+
\begin{pmatrix} \int_{\mathbb R} e^{itu} {a(u)}\, du& \int_{\mathbb R} e^{itu}{b(u)}\, du\\
-\int_{\mathbb R} e^{itu} \overline{ b(-u)} \, du & \int_{\mathbb R} e^{itu} \overline{a(-u)}\, du\end{pmatrix}
  \begin{pmatrix} \overline{L(it)}& -\overline{E(it)}\\
-\overline{D(it)} & \overline{C(it)}\end{pmatrix}\\
&+
\begin{pmatrix} \int_{\mathbb R} e^{itu} {a(u)}\, du& \int_{\mathbb R} e^{itu}{b(u)}\, du\\
-\int_{\mathbb R} e^{itu} \overline{ b(-u)} \, du & \int_{\mathbb R} e^{itu} \overline{a(-u)}\, du\end{pmatrix}
+  \begin{pmatrix} \overline{L(it)}& -\overline{E(it)}\\
-\overline{D(it)} & \overline{C(it)}\end{pmatrix}.
\end{split}
\]
We deduce that $C(it)=\overline{L(it)}$, $D(it)=-\overline{E(it)}$, $E(it)=-\overline{D(it)}$, $L(it)=\overline{C(it)}$ and
$$
H(it)=I_2+\begin{pmatrix} C(it)& D(it)\\
-\overline{D(it)} & \overline{C(it)}\end{pmatrix}.
$$
The function $1+(C(it)+D(it)j)=1+\int_{\mathbb R} e^{itu} (c(u)+d(u)j) du$, for some suitable $c(u), d(u)$ is, by construction the inverse of $F$ in $\mathcal W(\mathbb R,\mathbb H)$ and thus (i) holds.
This concludes the proof.
\end{proof}

Unlike in the discrete case, the notion of positivity and the associated factorization is slice dependent.
To illustrate this fact, fix some $i\in\mathbb S$ and consider the
function $F(p)=\frac{1-ip}{1+p^8}$. Then $F(it)=\frac{1+t}{1+t^8}$ belongs to the classical Wiener algebra (since it is real valued and has no
no poles on the real line; see for instance \cite{bgk2}) and there exists $M>0$ such that $G(t)=M+F(it)>0$ for all $t\in\mathbb R$. So the function $G(t)$ admits a spectral factorization. On the other hand for $j\in\mathbb S$ different from $\pm i$ the function
$F(jt)$ is not real valued, and, in particular, is not positive on the real line.

\begin{e-definition}
\label{def:August31tr1}
Fix $i \in \mathbb{S}$. We denote by $\mathcal{W}_+(i \mathbb{R},\mathbb H)$ and $\mathcal{W}_-(i \mathbb{R},\mathbb H)$ the set of elements $F(iy) = c + \int_{-\infty}^{\infty} e^{i y u} f(u) du$, with $F\in\mathcal W(\mathbb{R},\mathbb H))$
so that $f(u) = 0$ for $u < 0$ and $f(u) = 0$ for $u > 0$, respectively.
\end{e-definition}

It is immediate that $\mathcal W_+(i \mathbb{R},\mathbb H)$ and $\mathcal W_{-}(i \mathbb{R},\mathbb H)$ are subalgebras of $\mathcal W(\mathbb{R},\mathbb H)$.

\begin{theorem}
Let $F\in\mathcal W_+(i\mathbb R,\mathbb H)$. Then the following are equivalent{\rm :}
\begin{enumerate}
\item[(i)] $F$ is invertible in $\mathcal W_+(i \mathbb{R},\mathbb H)$ (with inverse denoted by $F^{-1}$){\rm ;}

\item[(ii)] $\omega(F)$ is invertible in $\mathcal W_+^{2\times 2}(\mathbb R)$.

\end{enumerate}
\end{theorem}

\begin{proof} Assume that (i) holds. Then $\omega(F)(it)\omega(F^{-1})(it)=I_2$. Since both functions
$\omega(F)(it)$ and $\omega(F^{-1})(it)$ have analytic extensions (as functions of $t$) to the open left half-plane, (ii) follows.\smallskip

Assume now that (ii) holds.Then the same arguments as in the proof of Theorem \ref{theorem43} (with integrals on $[0,\infty)$ instead of $\mathbb R$) will show that $(i)$ holds.
\end{proof}

We note that if $F\in\mathcal W(\mathbb R,\mathbb H)$ is strictly positive on $i\mathbb R$ it admits a spectral factorization with factor in
$\mathcal W_+(i\mathbb R,\mathbb H)$. The arguments are similar and omitted.

\section{Hardy spaces and Wiener-Hopf operators}
\setcounter{equation}{0}
We will begin by recalling some classical facts from the theory of Hardy spaces. Let $i\in\mathbb S$. We shall identify as usual the complex
plane $\mathbb C_i$ with $\mathbb C$. Let $\Pi_+ = \{ z \in \mathbb{C}: {\rm Re}\; z > 0 \}$. The Hardy space $H_2(\Pi_+)$ consists of functions $f$ which are holomorphic on $\Pi_+$ and satisfy
$$\sup_{x > 0} \int_{-\infty}^{\infty} |f(x+iy)|^2 dy < \infty.$$
Recall that $f \in H_2(\Pi_+)$ has nontangential limits $f(iy)$ for almost all (with respect to the Lebesgue measure) $iy$, $f|_{i \mathbb{R}} \in L_2(\mathbb{R})$ and
$$\sup_{x > 0} \int_{-\infty}^{\infty} |f(x + iy)|^2 dy = \int_{-\infty}^{\infty} |f(iy)|^2 dy.$$
An important fact is that $f \in H_2(\Pi_+)$ if and only if there exists $\tilde{f} \in L_2(\mathbb{R})$ such that
\begin{equation}
\label{eq:August27uy1}
f(z) = \int_0^{\infty} e^{i z x} \tilde{f}(x) dx \quad {\rm for} \quad z \in \Pi_+.
\end{equation}

Upon identifying functions $f \in H_2(\Pi_+)$ with their boundary values, it can be checked that $H_2(\Pi_+)$ is a closed subspace of $L_2(i \mathbb{R})$, i.e., the Hilbert space of complex-valued functions $f$ defined on $i \mathbb{R}$ so that
$$\int_{-\infty}^{\infty} |f (i y)|^2 dy < \infty.$$
The inner product $\langle \cdot , \cdot \rangle_{L_2(i \mathbb{R})}$ is given by
$$\langle f, g \rangle_{L_2(i \mathbb{R})} = \int_{-\infty}^{\infty} \overline{g(i y )} f(i y) dy.$$
 It can also be checked that $H_2(\Pi_-) := H_2(\Pi_+)^{\perp}$ consists of functions $f$ which are holomorphic on $\Pi_- = \{ z \in \mathbb{C}: {\rm Re}\; z < 0 \}$ and satisfy
$$\sup_{x < 0 } \int_{-\infty}^{\infty} | f(x+iy) |^2 dy < \infty.$$
An important fact is that $f \in H_2(\Pi_-)$ if and only if there exists $\tilde{f} \in L_2(\mathbb{R})$ such that
\begin{equation}
\label{eq:August27gg1}
f(z) = \int_{-\infty}^0 e^{i z x} \tilde{f}(x) dx.
\end{equation}
Finally, upon identifying functions $H_2(\Pi_+)$ and $H_2(\Pi_-)$ via their boundary values with functions in $L_2(i \mathbb{R})$,
\begin{equation}
\label{eq:August27kk1}
L_2(i \mathbb{R} ) = H_2(\Pi_+) \oplus H_2(\Pi_-)
\end{equation}
follows from \eqref{eq:August27uy1} and \eqref{eq:August27gg1}.

Returning to the quaternionic setting, see \cite{kyp}, let $\mathbb{H}_+ = \{ p \in \mathbb{H}: {\rm Re}\; p > 0 \}$ and $\Pi_{+, i} = \mathbb{H}_+ \cap \mathbb{C}_i$ for any $i \in \mathbb{S}$. Let $f_i = f|_{\Pi_{+, i}}$, where $f$ is a function defined on $\mathbb{H}_+$. We let $H_2(\Pi_+, i)$ denote the space of slice hyperholomorphic functions on $\mathbb{H}_+$ so that
$$\int_{-\infty}^{\infty} |f_i(iy) |^2 dy < \infty,$$
where $f_i(iy)$ denote the nontangential limit at $iy$ of $f_i$.

Given $f \in H_2(\Pi_{+, i} )$, the restriction $f_i$ can be written as
\begin{equation}
\label{eq:August28qd55}
f_i(x+iy) = f_1(x+iy) + f_2(x + iy) j,
\end{equation}
where $j \in \mathbb{S}$ so that $ij = -1$, i.e., $i$ and $j$ are orthogonal and $f_1, f_2$ are $\mathbb{C}_i$-valued holomorphic functions. Note that
\begin{equation}
\label{eq:August28tj1}
f_i(iy) = f_1(iy) + f_2(iy) j \quad {\rm a.e.}
\end{equation}
and
\begin{equation}
\label{eq:August28yy2}
|f_i(x+iy) |^2 = |f_1(x+iy)|^2 + |f_2(x+iy)|^2.
\end{equation}

Thus, making use of \eqref{eq:August28tj1} and \eqref{eq:August28yy2} we get
\begin{align}
 \sup_{x > 0 } \int_{-\infty}^{\infty} |f_i(x+iy)|^2 dy =& \; \int_{-\infty}^{\infty} |f_1(iy)|^2 dy + \int_{-\infty}^{\infty} |f_2(iy)|^2 dy \nonumber \\
 =& \; \int_{-\infty}^{\infty} |f_i(iy)|^2 dy. \label{eq:August28uu1}
\end{align}

Let $j \in \mathbb{S}$ be orthogonal to $i \in \mathbb{S}$. We let $L_2(i \mathbb{R}, \mathbb{H})$ consist of all functions $g$ of the form
$$g(iy) = g_1(iy) + g_2(iy) j,$$
where $g_1$ and $g_2$ belong to $L_2(i \mathbb{R} )$, with norm
$$
\| g \|_{L_2(i \mathbb{R}, \mathbb{H})} = (\| g_1 \|_{L_2(i \mathbb{R})}^2 + \| g_2 \|_{L_2(i \mathbb{R})}^2)^{1/2}
$$

Summarizing the above discussion we arrive at Theorem \ref{thm:August28ii6} and \ref{thm:August28gb1}.

\begin{theorem}
\label{thm:August28ii6}
Let $j \in \mathbb{S}$ be orthogonal to $i \in \mathbb{S}$. The space $H_2(\Pi_{+,i})$ is a closed subspace of $L_2(i \mathbb{R}, \mathbb{H})$, upon identifying functions in $H_2(\Pi_{+, i})$ with their boundary values. Moreover,
\begin{equation}
\label{eq:August28rr1}
L_2(i \mathbb{R}, \mathbb{H}) = H_2(\Pi_{+, i}) \oplus H_2(\Pi_{+, i})^{\perp}.
\end{equation}
\end{theorem}

\begin{theorem}
\label{thm:August28gb1}
Fix $i \in \mathbb{S}$. A function $f \in H_2(\Pi_{+,i})$ if and only if $f_1 \in H_2(\Pi_+)$ and $f_2 \in H_2(\Pi_+)$, where $f_1$ and $f_2$ are as in \eqref{eq:August28qd55}.
\end{theorem}

\begin{corollary}
\label{cor:August28uyb10}
Fix $i \in \mathbb{S}$. A function $f \in H_2(\Pi_{+,i})$ if and only if
\begin{equation}
\label{eq:August28un1}
f_i(x+iy) = \int_0^{\infty} e^{i (x+iy) u} \tilde{f}_1(u) du + \left( \int_0^{\infty} e^{i (x+iy) u } \tilde{f}_2(u) du \right) j,
\end{equation}
where $x > 0$, $\tilde{f}_1$ and $\tilde{f}_2$ belong to $L_2(\mathbb{R})$ and $j \in \mathbb{S}$ is orthogonal to $i \in \mathbb{S}$.
\end{corollary}

\begin{proof}
Assertion \eqref{eq:August28un1} follows easily from Theorem \ref{thm:August28gb1} and \eqref{eq:August27uy1}.
\end{proof}

With these definitions and results at hand we turn to an application to the theory of Wiener-Hopf operators.
Fix $i \in \mathbb{S}$ and let $j \in \mathbb{S}$ be orthogonal to $i$, i.e., $ij = -1$. Let $P_i$ and $Q_i$ denote the orthogonal projections of $L_2(i \mathbb{R}, \mathbb{H})$ onto $H_2(\Pi_{+, i})$ and $H_2(\Pi_{+, i})^{\perp}$. Note that $P_i$ and $Q_i$ are well-defined due to Theorem \ref{thm:August28ii6}. If $ \Phi \in \mathcal{W}(\mathbb{R},\mathbb H)$ and $F \in H_2(\Pi_{+, i})$, then there exist a constant $c_1$, $\phi \in L_1(\mathbb{R})$ and $f \in L_2(i \mathbb{R})$ so that
$$\Phi(iy) = c_1 + \int_{-\infty}^{\infty} e^{iyu} \phi(u) du$$
and
$$F(iy) = \int_0^{\infty} e^{i yu} f(iu) du \quad {\rm a.e.}.$$
Thus, we can define
$$(\Phi \star F)(iy) = c_1 \int_0^{\infty} e^{i y u} f(u) du + \int_{-\infty}^{\infty} e^{i y u } (\phi \circ f)(u) du,$$
where
$$(\phi \circ f)(u) = \int_{-\infty}^{\infty} \phi(u-v) f(v) dv.$$
Since $F \in H_2(\Pi_{+, i}) \in H_2(\Pi_{+, i}) \subset L_2(i \mathbb{R}, \mathbb{H})$, we have that $\Phi \star F \in L_2(i \mathbb{R}, \mathbb{H})$ and hence
$$M_{\Phi} F := \Phi \star F$$
defines a multiplication operator from $H_2(\Pi_{+, i})$ to $L_2(i \mathbb{R}, \mathbb{H})$.

We are now ready to define the Wiener-Hopf operator based on $\Phi \in \mathcal{W}(\mathbb{R},\mathbb H)$. The operator $T_{\Phi}: H_2(\Pi_{+, i}) \to H_2(\Pi_{+, i})$ given by
\begin{equation}
\label{eq:August31rr1}
T_{\Phi} F = P_i M_{\Phi} F \quad {\rm for} \quad F \in H_2(\Pi_{+, i})
\end{equation}
is the {\it Wiener-Hopf operator} based on $\Phi \in \mathcal{W}(\mathbb{R},\mathbb H)$.

\begin{theorem}
\label{thm:August31yyg1}
Fix $i \in \mathbb{S}$. If $\Phi, \Psi \in \mathcal{W}( \mathbb{R},\mathbb H)$, then $T_{\Phi} T_{\Psi}$ is a Wiener-Hopf operator if and only if $\Phi \in \mathcal{W}_-(i \mathbb{R},\mathbb H)$ or $\Psi \in \mathcal{W}_+(i \mathbb{R},\mathbb H)$. In this case, $T_{\Phi} T_{\Psi} = T_{\Phi \star \Psi}$.
\end{theorem}

\begin{proof}
Theorem \ref{thm:August31yyg1} is proved in much the same way as Theorem \ref{theorem:Apr27uz1}.
\end{proof}

\bibliographystyle{plain}



\end{document}